\documentclass{article}
  
\usepackage{latexsym} 
\usepackage{amsmath}
\usepackage{amsthm}
\usepackage{amssymb}
\usepackage{verbatim} 
\usepackage{dsfont}

\newtheorem{Counter}{!!!!!!!}

\newtheorem{Thm}[Counter]{Theorem}

\newtheorem{Lem}[Counter]{Lemma}

\begin{document}

\title {The optimal free knot spline approximation of stochastic differential equations with additive noise}

\author{
Mehdi Slassi\thanks{Fachbereich Mathematik, Technische Universit\"at
Darmstadt, Schlo\ss gartenstra\ss e 7, 64289 Darmstadt, Germany}
}

\date{\today}

\maketitle

\begin{abstract}
In this paper we analyse the pathwise approximation of stochastic differential equations by polynomial splines with free knots. The pathwise distance between the solution and its approximation is measured globally on the unit interval in the $L_{\infty}$-norm, and we study the expectation of this distance. For equations with additive noise we obtain sharp lower and upper bounds for the minimal error in the class of arbitrary spline approximation methods, which use $k$ free knots. 
The optimal order is achieved by an approximation
method $\widehat{X}_{k}^{\dagger}$, which combines an Euler scheme on a coarse grid with an optimal
spline approximation of the Brownian motion $W$  with $k$
free knots.

\medskip

\end{abstract}
\section{Introduction }

\let\languagename\relax 

Consider a scalar stochastic differential equation (SDE) with additive noise
\begin{equation}
dX\left(t\right)=a\left(t,X\left(t\right)\right)dt+\sigma\left(t\right)dW\left(t\right),\qquad t\in \left[0,1\right],\label{eq1}\end{equation}
with initial value $X(0)$. Here $W=\left(W\left(t\right)\right)_{t\geq 0}$
denotes a one-dimensional Brownian motion on a probability space $\left(\Omega,\mathcal{F},\mathbb{P}\right)$. We study pathwise approximation of equation (\ref{eq1}) on the unit interval by polynomial splines with free knots.

Let $X$ and $\widehat{X}$ denote the strong solution and an approximate solution on $\left[0,1\right]$, respectively. For the pathwise error we consider the distance
in $L_{\infty}$-norm\[
\bigl\Vert X-\widehat{X}\bigr\Vert_{L_{\infty}\left[0,1\right]}=\sup_{0\leq t\leq1}\bigl|X\left(t\right)-\widehat{X}\left(t\right)\bigr|,\]
 and we define the error $e_{q}\bigl(\widehat{X}\bigr)$ of the approximation
$\widehat{X}$ by averaging over all trajectories, i.e., \begin{equation}
e_{q}\bigl(\widehat{X}\bigr)=\left(E^{*}\bigl\Vert X-\widehat{X}\bigr\Vert_{L_{\infty}\left[0,1\right]}^{q}\right)^{1/q},\qquad  1\leq q<\infty.\label{eq2}\end{equation}
Here we use the outer expectation value $E^{*}$ in order to avoid
cumbersome measurability considerations. The reader is referred to
\cite{WellnerVanderVaart:96} for a detailed study of the outer integral
and expectation.
In the sequel, for two sequences $\left(a_{k}\right)_{k\in\mathbb{N}}$
and $\left(b_{k}\right)_{k\in\mathbb{N}}$ of positive real numbers
we write $a_{k}\approx b_{k}$ if $\lim_{k\to\infty}a_{k}/b_{k}=1$ 
and $a_{k}\gtrsim b_{k}$ if $\liminf_{k\to\infty}a_{k}/b_{k}\geq1$.
Additionally $a_{k}\asymp b_{k}$ means $C_{1}\leq a_{k}/b_{k}\leq C_{2}$
for all $k\in\mathbb{N}$ and some positive constants $C_{i}$.

Typically, piecewise linear functions with fixed knots or with sequential selection of knots are used to approximate the solution of SDEs globally on a time interval, and approximations of this kind are considered in the present paper as particular cases, too. 

For $k\in \mathbb{N}$  we use $\widehat{X}_{k}^{e}$  to denote the piecewise interpolated Euler scheme with constant step-size  $1/k$. In \cite{HofmannGronbachRitter:00}  Hofmann et al. 
have determined the strong asymptotic behaviour of $e_{q}\left(\widehat{X}_{k}^{e}\right)$
with an explicitly given constant, namely
\begin{equation}
e_{q}\left(\widehat{X}_{k}^{e}\right)\approx \frac{C_{e}}{\sqrt{2}}\cdot\left(\ln k/k\right)^{1/2}\label{order}
\end{equation}
with \[
C_{e}=\left\Vert \sigma\right\Vert _{L_{\infty}\left[0,1\right]},\]
where $\left\Vert \sigma\right\Vert _{L_{\infty}\left[0,1\right]}=\sup_{t\in\left[0,1\right]}\left|\sigma\left(t\right)\right|.$ Note that the upper
bound in (\ref{order}) has first been given in \cite{Faure:01} with an unspecified constant.

Now, we recall known results concerning the approximations that are based on a sequential
selection of knots to evaluate $W$, see \cite{HofmannGronbachRitter:00, Gronbach:02}
for a formal definition of such methods. This includes numerical methods with adaptive step size control. In \cite{HofmannGronbachRitter:00} Hofmann et al. show that a step size proportional to the inverse of the current value of $\sigma^{2}$  leads to an asymptotically optimal method $\widehat{X}_{k}^{a}$, more precisely 
\begin{equation}
e_{q}\left(\widehat{X}_{k}^{a}\right)\approx \frac{C_{a}}{\sqrt{2}}\cdot\left(\ln k/k\right)^{1/2}\label{klaus1}
\end{equation}
 and \[
C_{a}=\left\Vert \sigma\right\Vert _{2},\]
where $\left\Vert \sigma\right\Vert _{2}=\left(\int_{0}^{1}\left(\sigma\left(t\right)\right)^{2}\, dt\right)^{1/2}.$
Moreover, they establish strong asymptotic optimality of the sequence
$\widehat{X}_{k}^{a}$, i.e., for every sequence of methods $\widehat{X}_{k}$
that use $k$ sequential observations of $W$ \begin{equation}
e_{q}\left(\widehat{X}_{k}\right)\gtrsim \frac{C_{a}}{\sqrt{2}}\cdot\left(\ln k/k\right)^{1/2}.\label{klaus2}\end{equation}
Typically, $C_{a}<C_{e}$ and $C_{a}>0$, so that the convergence order $ \left(\ln k/k\right)^{1/2}$ cannot be improved by sequential observation of $W.$ A generalization of the results (\ref{order}), (\ref{klaus1}) and (\ref{klaus2}) to the case of systems of equations with multiplicative noise has been achieved in \cite{Gronbach:02}.

In the present paper we do not impose any restriction on the selection
of the knots.

For $k\in\mathbb{N}$ and $r\in\mathbb{N}_{0}$ we let $\Pi_{r}$
denote the set of polynomials of degree at most $r$, and we consider
the space $\Phi_{k,r}$ of polynomial splines $\varphi$ of degree
at most $r$ with $k-1$ free knots, i.e., \[
\varphi=\sum_{j=1}^{k}\mathbf{1}_{\left]t_{j-1},\; t_{j}\right]}\cdot\pi_{j},\]
where $0=t_{0}<\cdots<t_{k}=1$ and $\pi_{1},\ldots,\pi_{k}\in\Pi_{r}$. Note that the spline $\varphi$ uses $k+1$ knots, whereof $k-1$ can be chosen freely.
Then, any approximation method $\widehat{X}_{k}$ by splines with
$k-1$ free knots can be thought of as a mapping \[
\widehat{X}_{k}\;:\Omega\longrightarrow\Phi_{k,r},\]
and we denote this class of mappings by $\mathfrak{N}_{k,r}$.

Furthermore, we define the minimal error
\begin{equation} \label{rev1}
e_{k,q}^{\min}\left(X\right) = \inf \{e_{q}\bigl(\widehat{X}_{k}\bigr) : \widehat{X}_{k} \in \mathfrak{N}_{k,r} \},
\end{equation} 
i.e., the $q$-average $L_{\infty}$-distance of the solution $X$ to the spline space  $\Phi_{k,r}$. We shall study the strong asymptotic behaviour of $e_{k,q}^{\min}\left(X\right)$ as $k$ tends to infinity.

Note that spline approximation with free knots is a nonlinear approximation problem in the sense that the approximants do not come from linear spaces but rather from nonlinear manifolds $\Phi_{k,r}$. Nonlinear approximation
for deterministic functions has been extensively studied in the literature, see \cite{Devore:98}
for a survey. In the context of stochastic processes much less is
known, and we refer the reader to \cite{CohenD'Ales:97, CohenDaubechiesGuleryuz:02, CreutzigGronbachRitter:07, konPlaskota:05, Slassi:12}. At first in \cite{konPlaskota:05} and thereafter in \cite{CreutzigGronbachRitter:07, Slassi:12}
approximation by splines with free knots is studied, while wavelet       
methods are employed in \cite{CohenD'Ales:97,CohenDaubechiesGuleryuz:02}.

From Creutzig et al. \cite{CreutzigGronbachRitter:07} we know, that\begin{equation}
e_{k,q}^{\min}\left(X\right)\asymp\left(1/k\right)^{1/2}.\label{klaus3}\end{equation}
Hence free knot spline approximation yields a better rate of convergence than (\ref{order}) and (\ref{klaus1}). We add, that the same order
of convergence is achieved by the average Kolmogorov widths, see \cite{Creutig:02,Mairove:93,Mairove:96},
but asymptotically optimal subspaces seem to be unknown.

In \cite{Slassi:12} we analyse an approximation method $\widehat{X}_{k}^{*}$, which achieves the convergence order $1/\sqrt{k}$. The method $\widehat{X}_{k}^{*}$ combines a Milstein scheme on a coarse grid with an optimal spline approximation of the Brownian motion $W$. The approximation method  $\widehat{X}_{k}^{*}$ basically works in two steps. First, we take the Milstein scheme to estimate the drift and diffusion coefficients at equidistant discrete points $t_{\ell}$. At the second stage we piecewise freeze the drift and diffusion coefficients and we consider on each subinterval $\left[t_{\ell-1},t_{\ell}\right]$  the asymptotically optimal spline approximation of the Brownian motion $W\left(t\right)-W\left(t_{\ell}\right)$ with equal number of free knots fixed a priori. For adaptive step size
control a similar idea has been used in \cite{HofmannGronbachRitter:01}. In the particular case of SDEs with additive noise, we show that the error of $\widehat{X}_{k}^{*}$ satisfies 
\begin{equation}\label{Tina0}
e_{q}\bigl(\widehat{X}_{k}^{*}\bigr)\approx\left(E\left(\tau_{1,1}\right)\right)^{-1/2}\cdot C_{e}\cdot\left(1/k\right)^{1/2},
\end{equation}
where 
$$\tau_{1,1}=\inf\Bigl\{ t>0\;\mid\;\inf_{\pi\in\Pi_{r}}\left\Vert W-\pi\right\Vert _{L_{\infty}\left[0,t\right]}>1\Bigr\}.$$
Hence the stopping time $\tau_{1,1}$ yields the maximal length of 
a subinterval $\left[0,t\right]$ that permits best approximation of $W$ 
by polynomials of degree at most $r$ with error at most one.  

In order to improve the asymptotic constant in (\ref{Tina0}) we introduce in the present paper an approximation method  $\widehat{X}_{k}^{\dagger}$. The method $\widehat{X}_{k}^{\dagger}$  is defined in the same way as $\widehat{X}_{k}^{*}$,  where the number of free knots used in each subinterval $\left[t_{\ell-1},t_{\ell}\right]$ is roughly  proportional to $\left(\sigma\left(t_{\ell-1}\right)\right)^{2}$ .
For the error of $\widehat{X}_{k}^{\dagger}$  we establish the strong asymptotic behaviour with an explicitly given constant, namely
\begin{equation}
 e_{q}\left(\widehat{X}_{k}^{\dagger}\right)\approx\left(E\left(\tau_{1,1}\right)\right)^{-1/2}\cdot C_{a} \cdot\left(1/k\right)^{1/2}. \label{Tina2} 
\end{equation}
Note that the new approximation performs asymptotically better than the approximation  $\widehat{X}_{k}^{*}$ in many cases.

In \cite{CreutzigGronbachRitter:07} the lower and  upper bound in (\ref{klaus3})
are proven non-constructively and the method of proof does not allow
to control asymptotic constants. In this paper we wish to find sharp lower and upper bounds for the minimal error (\ref{rev1}) for SDEs with additive noise. We show that the minimal errors satisfy 
 \begin{equation} e_{k,q}^{\min}\left(X\right)\approx\left(E\left(\tau_{1,1}\right)\right)^{-1/2}\cdot C_{a}\cdot\left(1/k\right)^{1/2}. \label{Tina4} 
\end{equation}
We note that the order of convergence in (\ref{Tina2}) and (\ref{Tina4}) does not depend on the degree $r$ of the approximation splines. The parameter $r$ has only an impact on the asymptotic constant $ E\left(\tau_{1,1}\right)$. We add that
due to (\ref{Tina2}) and (\ref{Tina4})  the method  $\widehat{X}_{k}^{\dagger}$  is asymptotically optimal in the class $\mathfrak{N}_{k,r}$ for every equation (\ref{eq1}) with additive noise.

The structure of the paper is as follows. In Section 2 we specify our assumptions regarding the equation (\ref{eq1}). The drift and diffusion coefficients must satisfy Lipschitz conditions, and the initial value must have a finite $q$-moment for all $q\geq 1$. Moreover, we briefly recall
some definitions and results from \cite{CreutzigGronbachRitter:07}
concerning the optimal approximation of $W$ by polynomial splines
with free knots. We introduce the approximation method
$\widehat{X}_{k}^{\dagger}$  and state the main results. Proofs are given in Section 3.

\section{Main result}

Given $\varepsilon>0,$ we define a sequence of stopping times by
$\tau_{0,\varepsilon}=0$ and \[
\tau_{j,\varepsilon}=\tau_{j,\varepsilon}\left(W\right)=\inf\Bigl\{ t>\tau_{j-1,\varepsilon}\;\mid\;\inf_{\pi\in\Pi_{r}}\left\Vert W-\pi\right\Vert _{L_{\infty}\left[\tau_{j-1,\varepsilon},\; t\right]}>\varepsilon\Bigr\},\quad j\geq1.\]
 For $j\in\mathbb{N}$ we define\[
\xi_{j,\varepsilon}=\tau_{j,\varepsilon}-\tau_{j-1,\varepsilon}.\]
These random variables yield the lengths of consecutive maximal subintervals
that permit best approximation from the space $\Pi_{r}$ with error
at most $\varepsilon$. For every $\varepsilon>0$ the random variables
$\xi_{j,\varepsilon}$ form an i.i.d. sequence with \[
\xi_{j,\varepsilon}\stackrel{d}{=}\varepsilon^{2}\cdot\tau_{1,1}\quad\mathrm{and}\quad E\left(\tau_{1,1}^{m}\right)<\infty\]
for every $m\in\mathbb{N}$, see \cite{CreutzigGronbachRitter:07}.
Furthermore, we consider the pathwise minimal approximation error
by splines using $k-1$ free knots \[
\gamma_{k}=\gamma_{k}\left(W\right)=\inf\left\{ \varepsilon>0\;\mid\;\tau_{k,\varepsilon}\geq1\right\} .\]
An optimal spline approximation of $W$ on $\left[0,1\right]$
with $k-1$ free knots is given by\begin{equation}
\widetilde{W}_{k}=\sum_{j=1}^{k}\mathbf{1}_{\left]\tau_{j-1,\gamma_{k}},\;\tau_{j,\gamma_{k}}\right]}\cdot\;\mathrm{argmin}_{\pi\in\Pi_{r}}\left\Vert W-\pi\right\Vert _{L_{\infty}\left[\tau_{j-1,\gamma_{k}},\;\tau_{j,\gamma_{k}}\right]}.\label{eq3}\end{equation}
More precisely, from \cite{CreutzigGronbachRitter:07} we know that
\begin{equation}
\bigl\Vert W-\widetilde{W}_{k}\bigr\Vert_{L_{\infty}\left[0,\,1\right]}=\gamma_{k}\;\approx\;\left(E\left(\tau_{1,1}\right)\cdot k\right)^{-1/2}\qquad a.s.\label{eq5}\end{equation}
and\begin{equation}
\left(E^{*}\left(\bigl\Vert W-\widetilde{W}_{k}\bigr\Vert_{L_{\infty}\left[0,\,1\right]}^{q}\right)\right)^{1/q}\;\approx\;\left(E\left(\tau_{1,1}\right)\cdot k\right)^{-1/2}.\label{eq4}\end{equation}

We assume that the drift coefficient $a:\left[0,1\right]\times\mathbb{R}\to\mathbb{R}$ and the diffusion  coefficient $\sigma:\left[0,1\right]\to\mathbb{R}$
and the initial value $X\left(0\right)$ have the following properties.

\begin{itemize}
\item $\mathrm{(A)}$ $a$ is differentiable with respect
to the state variable. Moreover, there exists a constant $K>0$, such that \begin{eqnarray*}
\left|a\left(t,x\right)-a\left(t,y\right)\right| & \leq & K\cdot\left|x-y\right|,\\
\left|a\left(s,x\right)-a\left(t,x\right)\right| & \leq & K\cdot\left(1+\left|x\right|\right)\cdot\left|s-t\right|,\\
\left|a^{\left(0,1\right)}\left(t,x\right)-a^{\left(0,1\right)}\left(t,y\right)\right| & \leq & K\cdot\left|x-y\right|\end{eqnarray*}
for all $s,\; t\in\left[0,1\right]$ and $x,\; y\in\mathbb{R}.$ 
\item $\mathrm{(B)}$ There exists a constant $K>0$, such that
$$
\left|\sigma\left(s\right)-\sigma\left(t\right)\right| \leq K\cdot\left|s-t\right|  $$
  and  
$$ \left|\sigma\left(t\right)\right| >0 $$
for all $ s,\;t\in\left[0,1\right].$ 
\item $\mathrm{(C)}$ The initial value $X\left(0\right)$ is independent
of $W$ and \[
E\left(\left|X\left(0\right)\right|^{q}\right)<\infty\qquad\mathrm{for\; all}\; q\geq1.\]
\end{itemize}
Note that $\mathrm{(A)}$ yields the linear growth condition, i.e., there exists a constant $C > 0$ such that \begin{equation}
\left|a\left(t,x\right)\right|\leq C\cdot\left(1+\left|x\right|\right)\label{growth}\end{equation} for all $ t\in\left[0,1\right]$ and $x\in\mathbb{R}.$\\
Conditions $\mathrm{(A)}$ and $\mathrm{(C)}$  are standard assumptions for analysing stochastic differential equations, while $\mathrm{(B)}$ is slightly stronger than the standard assumption for equations with additive noise. We conjecture, that the weaker condition $\sigma \neq 0$ would be sufficient to obtain the results in the paper.
Given the above properties, a pathwise unique strong solution of equation
(\ref{eq1}) with initial value $X\left(0\right)$ exists. In particular
the conditions assure that\begin{equation}
E\left(\left\Vert X\right\Vert _{L_{\infty}\left[0,1\right]}^{q}\right)<\infty\qquad\mathrm{for\; all}\; q\geq1.\label{eqEX}\end{equation}

Next, we turn to the definition of the spline approximation scheme
$\widehat{X}_{k}^{\dagger}$. Fix $\delta\in\left(1/2,1\right)$ and for
$k\in\mathbb{N}$ take
\begin{equation}
n_{k}=\left\lfloor k^{\delta}\right\rfloor. \label{revneu1}
\end{equation}
Note that\begin{equation}
\lim_{k\to\infty}\frac{n_{k}}{k}=0\qquad\mathrm{and}\qquad\lim_{k\to\infty}\frac{\sqrt{k}}{n_{k}}=0.\label{asymto1}\end{equation}
We take the Euler scheme to compute an approximation to $X$ at
the discrete points
\begin{equation}
t_{\ell}=\frac{\ell}{n_{k}},\qquad\ell=0,\ldots,n_{k}.\label{discrtization1}
\end{equation}
This scheme is defined by 
\[
\check{X}\left(t_{0}\right)=X\left(0\right)\]
and
\begin{equation}
\check{X}\left(t_{\ell+1}\right) = \check{X}\left(t_{\ell}\right) +a\left(t_{\ell},\check{X}\left(t_{\ell}\right)\right)\cdot\left(t_{\ell+1}-t_{\ell}\right)+\sigma\left(t_{\ell}\right)\cdot\left(W\left(t_{\ell+1}\right)-W\left(t_{\ell}\right)\right).\label{euler1}
\end{equation}
For every $\ell\in\left\{ 0,\ldots,n_{k}-1\right\} $ we consider the Brownian
motion $W^{\ell}$, defined by\[
W^{\ell}\left(t\right)=W\left(t\right)-W\left(t_{\ell}\right),\qquad t\in\left[t_{\ell},t_{\ell+1}\right].\]
Put $$ \sigma_{\ell}= \sigma\left(t_{\ell}\right) $$ and let
$$ m_{\ell,k} = \left\lfloor \left(\sigma_{\ell}^{2}/\sum_{i=0}^{n_{k}-1}\sigma_{i}^{2}\right)\cdot\left(k-n_{k}\right)\right\rfloor +1.$$
Let $\widehat{W}_{m_{\ell,k}}^{\ell}$ denote the asymptotically
optimal spline approximation of $W^{\ell}$ on the interval $\left[t_{\ell},t_{\ell+1}\right]$
with $m_{\ell,k}-1$ free knots, cf. (\ref{eq3}). Now, the approximation
method $\widehat{X}_{k}^{\dagger}$ is given by\[
\widehat{X}_{k}^{\dagger}\left(t_{0}\right)=X\left(0\right)\]
and for $t\in\left]t_{\ell},t_{\ell+1}\right]$
\begin{equation}
\widehat{X}_{k}^{\dagger}\left(t\right)=\check{X}\left(t_{\ell}\right)+a\left(t_{\ell},\check{X}\left(t_{\ell}\right)\right)\cdot\left(t-t_{\ell}\right)+\sigma_{\ell}\cdot\widehat{W}_{m_{\ell,k}}^{\ell}\left(t\right).\label{eq7}\end{equation}
Note that the number of free knots on $\left]t_{\ell},t_{\ell+1}\right[$ is given by $m_{\ell,k}-1$.
Since $$  k-n_{k}\leq n_{k}+1 +\sum_{\ell =0}^{n_{k}-1} \left(m_{\ell,k}-1\right) \leq k+1,$$ the method
$\widehat{X}_{k}^{\dagger}$ uses at most $k+1$ knots for every trajectory. Due to (\ref{asymto1}) the upper bound $k+1$ is sharply asymptotical. By formally introducing a few additional knots we get a method with $k-1$ free knots, i.e., $\widehat{X}_{k}^{\dagger} \in \mathfrak{N}_{k,r}$.\\

Now we can state the main results of the paper. 
\begin{Thm}
\label{theorem1-1} Assume that $\mathrm{(A)}$, $\mathrm{(B)}$ and $\mathrm{(C)}$
hold for equation (\ref{eq1}). Then we have
\begin{equation}
\lim_{k\to\infty}\sqrt{k}\cdot e_{q}\left(\widehat{X}_{k}^{\dagger}\right)=\left(E\left(\tau_{1,1}\right)\right)^{-1/2}\cdot\left\Vert \sigma\right\Vert _{2}\label{statement2}\end{equation}for all $q\geq1.$
\end{Thm}
\textcompwordmark{}
\begin{Thm}
\label{theorem1-2}Assume that $\mathrm{(A)}$, $\mathrm{(B)}$ and $\mathrm{(C)}$
hold for equation (\ref{eq1}). Then, the minimal errors satisfy 
\begin{equation}
\lim_{k\to\infty}\sqrt{k}\cdot e_{k,q}^{\min}\left(X\right)=\left(E\left(\tau_{1,1}\right)\right)^{-1/2}\cdot\left\Vert \sigma\right\Vert _{2}\label{appro2}
\end{equation}
 for all $q\geq 1.$ 
\end{Thm}
Due to (\ref{statement2}) and (\ref{appro2}) the method $\widehat{X}_{k}^{\dagger}$ is asymptotically optimal in the class $\mathfrak{N}_{k,r}$ for every equation (\ref{eq1}) with additive noise.
\section{Proof of main result }
For the proof of Theorem \ref{theorem1-1} we need the following Lemma.\\
For every $\ell=0,\ldots,n_{k}-1$ we consider the pathwise minimal 
approximation error of $W^{\ell}$\[
\gamma_{m_{\ell,k}}^{\ell}=\gamma_{m_{\ell,k}}^{\ell}\bigl(W^{\ell}\bigr)=\inf\bigl\{\varepsilon>0\;\mid\;\tau_{m_{\ell,k},\varepsilon}^{\ell}\geq t_{\ell+1}\bigr\},\]
where $\left(\tau_{j,\varepsilon}^{\ell}\right)_{j\in\mathbb{N}}$
denotes the sequence of stopping times on $\left[t_{\ell},t_{\ell+1}\right],$
defined by\[
\tau_{0,\varepsilon}^{\ell}=t_{\ell}\] 
and\[
\tau_{j,\varepsilon}^{\ell}=\tau_{j,\varepsilon}^{\ell}\bigl(W^{\ell}\bigr)=\inf\Bigl\{ t>\tau_{j-1,\varepsilon}^{\ell}\;\mid\;\inf_{\pi\in\Pi_{r}}\bigl\Vert W^{\ell}-\pi\bigr\Vert_{L_{\infty}\left[\tau_{j-1,\varepsilon}^{\ell},\, t\right]}>\varepsilon\Bigr\},\qquad j\geq1.\]
So, we have \begin{equation}
\bigl\Vert W^{\ell}-\widehat{W}_{m_{\ell,k}}^{\ell}\bigr\Vert_{L_{\infty}\left[t_{\ell},\, t_{\ell+1}\right]}=\gamma_{m_{\ell,k}}^{\ell}\qquad a.s.\label{Wl1}\end{equation}
Renormalizing each interval $\left[t_{\ell},t_{\ell+1}\right]$ to
$\left[0,1\right]$ it can easily be shown that 
 \begin{equation}
\gamma_{m_{\ell,k}}^{\ell}\stackrel{d}{=}\frac{1}{\sqrt{n_{k}}}\cdot\gamma_{m_{\ell,k}}\label{eq6}\end{equation} and
\begin{equation}
 \gamma_{m_{\ell,k}}^{\ell} \approx\left(E\tau_{1,1}\right)^{-1/2}\cdot\frac{1}{\sqrt{m_{\ell,k}}\cdot \sqrt{n_{k}}}\qquad a.s. \label{rev2}
\end{equation}
for every $\ell \in \mathbb{N}_{0}$, by Lemma 8 in \cite{CreutzigGronbachRitter:07}.
Furthermore, due to (\ref{asymto1}) we have 
\begin{equation}
\left|\sigma_{\ell}\right|\cdot \gamma_{m_{\ell,k}}^{\ell} \approx\left(E\tau_{1,1}\right)^{-1/2}\cdot\left\Vert \sigma\right\Vert _{2} \cdot\left(1/k\right)^{1/2} \qquad \mathrm{a.s.} \label{eq6N}
\end{equation}
for every $\ell \in \mathbb{N}_{0}.$

From now on let $C$ denote unspecified positive constants, which
only depend on the constant $K$ from condition $\mathrm{(A)}$, as
well as on $a\left(0,0\right),\;\sigma\left(0,0\right)$ and $E\bigl|X\left(0\right)\bigr|^{q}$. 
\begin{Lem}
\label{lemma2} For all $q \geq 1$ we have 
\begin{equation}
\lim_{k\to\infty}\sqrt{k}\cdot\left(E\max_{0\leq\ell\leq n_{k}-1}\left(\left|\sigma_{\ell}\right|\cdot \gamma_{m_{\ell,k}}^{\ell}\right)^{q}\right)^{1/q}=\left(E\left(\tau_{1,1}\right)\right)^{-1/2}\cdot \left\Vert \sigma\right\Vert _{2}.\label{eqlemma2}
\end{equation}
\end{Lem}
\begin{proof}
We have
$$ E\max_{0\leq\ell\leq n_{k}-1}\left(\left|\sigma_{\ell}\right|\cdot \gamma_{m_{\ell,k}}^{\ell}\right)^{q} = \left(\sum_{i=0}^{n_{k}-1}\sigma_{i}^{2}\right)^{q/2}\cdot E\max_{0\leq\ell\leq n_{k}-1}\left(\left(\left|\sigma_{\ell}\right|/\left(\sum_{i=0}^{n_{k}-1}\sigma_{i}^{2}\right)^{1/2}\right)\cdot \gamma_{m_{\ell,k}}^{\ell}\right)^{q}.$$ 
Let $\rho>1$, and put $\mu=E\bigl(\tau_{1,1}\bigr)$ and $a_{k}=\frac{1}{\left(\sqrt{n_{k}\cdot \left(k-n_{k}\right)}\right)^{q}\cdot\left(\sqrt{\mu/\rho}\right)^{q}}$. Then, \[
E\max_{0\leq\ell\leq n_{k}-1}\left(\left(\left|\sigma_{\ell}\right|/\left(\sum_{i=0}^{n_{k}-1}\sigma_{i}^{2}\right)^{1/2}\right)\cdot \gamma_{m_{\ell,k}}^{\ell}\right)^{q} \leq a_{k}+I\left(k\right),\]
where\[
I\left(k\right)=\int_{a_{k}}^{\infty}\mathbb{P}\left( \max_{0\leq\ell\leq n_{k}-1}\left(\left(\left|\sigma_{\ell}\right|/\left(\sum_{i=0}^{n_{k}-1}\sigma_{i}^{2}\right)^{1/2}\right)\cdot \gamma_{m_{\ell,k}}^{\ell}\right)^{q} >t\right)\, dt.\]
Firstly, by (\ref{asymto1}) we have\begin{equation}
\lim_{k\to\infty}\left(\sqrt{k}\right)^{q}\cdot\left(\sum_{i=0}^{n_{k}-1}\sigma_{i}^{2}\right)^{q/2}\cdot a_{k}=\frac{\left\Vert \sigma\right\Vert _{2}^{q}}{\left(\sqrt{\mu/\rho}\right)^{q}}.\label{eqI}\end{equation}
Using (\ref{eq6}) we get the estimate
$$I\left(k\right)\leq \sum_{\ell=0}^{n_{k}-1} \int_{a_{k}}^{\infty}\mathbb{P}\left(\gamma_{m_{\ell,k}}>t^{1/q}\cdot\sqrt{n_{k}} \cdot \left(\left(\sum_{i=0}^{n_{k}-1}\sigma_{i}^{2}\right)^{1/2}/\left|\sigma_{\ell}\right|\right)\right)\, dt.$$ Then,  we split the above right-hand side in to term $I_{1}\left(k\right)$ and $I_{2}\left(k\right)$, where
\[
I_{1}\left(k\right)= \sum_{\ell=0}^{n_{k}-1} \int_{a_{k}}^{n_{k}^{q}\cdot a_{k}}\mathbb{P}\left(\gamma_{m_{\ell,k}}>t^{1/q}\cdot\sqrt{n_{k}} \cdot \left(\left(\sum_{i=0}^{n_{k}-1}\sigma_{i}^{2}\right)^{1/2}/\left|\sigma_{\ell}\right|\right)\right)\, dt\]
and
$$ I_{2}\left(k\right) = \sum_{\ell=0}^{n_{k}-1} \int_{n_{k}^{q}\cdot a_{k}}^{\infty}\mathbb{P}\left(\gamma_{m_{\ell,k}}>t^{1/q}\cdot\sqrt{n_{k}} \cdot \left(\left(\sum_{i=0}^{n_{k}-1}\sigma_{i}^{2}\right)^{1/2}/\left|\sigma_{\ell}\right|\right)\right)\, dt.$$
We put\[ 
S_{n}=\sum_{j=1}^{n}\xi_{j,1}.\] 
Using the fact that for all $\varepsilon>0$ 
\begin{equation}
\mathbb{P}\left(\gamma_{m_{\ell,k}}\leq\varepsilon\right)=\mathbb{P}\left(S_{m_{\ell,k}}\geq1/\varepsilon^{2}\right) \label{rev3}
\end{equation}
and the random variables $\xi_{j,1}$ form an i.i.d. sequence (see \cite{CreutzigGronbachRitter:07}), it follows by substitution on the one hand that$$
I_{1}\left(k\right) = \frac{q}{2\left(\sqrt{\left(k-n_{k}\right)\cdot n_{k}}\right)^{q}}\sum_{\ell=0}^{n_{k}-1} \int_{\mu/\left(\rho\cdot n_{k}^{2}\right)}^{\mu/\rho}t^{-\left(q/2+1\right)}\cdot\mathbb{P}\left(\frac{S_{m_{\ell,k}}}{m_{\ell,k}}<t\right)\, dt.$$
For $\mu/\left(\rho\cdot n_{k}^{2}\right)\leq t\leq\mu/\rho$ we use H\"offding's
inequality to obtain\begin{eqnarray*}
t^{-\left(q/2+1\right)}\cdot\mathbb{P}\left(\frac{S_{m_{\ell,k}}}{m_{\ell,k}}<t\right) & \leq & t^{-\left(q/2+1\right)}\cdot\mathbb{P}\left(\left|\frac{S_{m_{\ell,k}}}{m_{\ell,k}}-\mu\right|>\mu-\mu/\rho\right)\\
 & \leq & \frac{\left(\rho\cdot n_{k}^{2}\right)^{q/2+1}}{\mu^{q/2+1}}\cdot2\exp\left(-2m_{\ell,k}.\left(\mu-\mu/\rho\right)^{2}\right)\end{eqnarray*} for every $\ell =0,\dots,n_{k}-1.$
This yields \begin{equation}
\lim_{k\to\infty}\left(\sqrt{k}\right)^{q}\cdot \left(\sum_{i=0}^{n_{k}-1}\sigma_{i}^{2}\right)^{q/2}\cdot I_{1}\left(k\right)=0.\label{eqJ1}\end{equation}
To verify this, it suffices to show
\begin{equation}
\lim_{k\to \infty}n_{k}^{q+2}\sum_{\ell=0}^{n_{k}-1}\exp\left(-2m_{\ell,k}\cdot c\right) = 0 \label{rev6}
\end{equation}
with $c>0$. In fact we have
$$ m_{\ell,k} \approx \frac{\sigma_{\ell}^{2}}{\left\Vert \sigma\right\Vert _{2}^{2}}\cdot \frac{k}{n_{k}}$$
for every $\ell \in \mathbb{N}_{0}$. Let $\alpha = \inf_{0\leq t \leq 1} \left(\sigma\left(t\right)\right)^{2}$. Using the definition of $n_{k}$ in (\ref{revneu1}) we get for $k$ sufficiently large
\begin{eqnarray*}
n_{k}^{q+2}\sum_{\ell=0}^{n_{k}-1}\exp\left(-2m_{\ell,k}\cdot c\right) &\leq & n_{k}^{q+2}\sum_{\ell=0}^{n_{k}-1}\exp\left(-\frac{\sigma_{\ell}^{2}}{\left\Vert \sigma\right\Vert _{2}^{2}}\cdot \frac{k}{n_{k}} \cdot c\right)\\
                                                                  &\leq&   k^{\delta \cdot \left(q+3\right)}\cdot \exp\left(-\frac{\alpha}{\left\Vert \sigma\right\Vert _{2}^{2}}\cdot k^{1-\delta} \cdot c\right),
\end{eqnarray*} 
which yields (\ref{rev6}).\\
On the other hand, using (\ref{rev3}) we obtain 
\begin{eqnarray*}
I_{2}\left(k\right) & =& \sum_{\ell=0}^{n_{k}-1}\int_{n_{k}^{q}\cdot a_{k}}^{\infty}\mathbb{P}\left(\sum_{j=1}^{m_{\ell,k}} \xi_{j,1}<\frac{\sigma_{\ell}^{2}}{t^{2/q}\cdot n_{k}\cdot\sum_{i=0}^{n_{k}-1}\sigma_{i}^{2} }\right)\, dt\\
                   & \leq & \sum_{\ell=0}^{n_{k}-1}\int_{n_{k}^{q}\cdot a_{k}}^{\infty}\left(\mathbb{P}\left(\tau_{1,1}<\frac{\sigma_{\ell}^{2}}{t^{2/q}\cdot n_{k}\cdot\sum_{i=0}^{n_{k}-1}\sigma_{i}^{2} }\right)\right)^{m_{\ell,k}}\, dt.\end{eqnarray*} 
Note that for all $\eta\leq1$ \[
\mathbb{P}\left(\tau_{1,1}\leq\eta\right)\leq\exp\left(-C\cdot\eta^{-1}\right)\]
with some constant $C>0$; see the proof of Lemma 8 in \cite{CreutzigGronbachRitter:07}. From (\ref{asymto1}) we have  $$\frac{k-n_{k}}{n_{k}^{2}}\cdot\left(\mu/\rho\right)\leq1$$ for $k$ sufficiently large. 
Then, for all $t\geq n_{k}^{q}\cdot a_{k}$ we have\[
 \frac{\sigma_{\ell}^{2}}{t^{2/q}\cdot n_{k}\cdot\sum_{i=0}^{n_{k}-1}\sigma_{i}^{2} }\leq1 \]for every $\ell =0,\dots,n_{k}-1.$\\
Hence, we get\begin{eqnarray*}
I_{2}\left(k\right) & \leq & n_{k}\cdot\int_{n_{k}^{q}\cdot a_{k}}^{\infty}\exp\left(-C\cdot\left(k-n_{k}\right)\cdot n_{k}\cdot t^{2/q}\right)\! dt\\
 & = & \frac{q}{2}\cdot n_{k}\cdot\frac{1}{\left(\sqrt{n_{k}\cdot \left(k-n_{k}\right)}\right)^{q}}\int_{\frac{n_{k}^{2}}{\mu/\rho}}^{\infty}t^{q/2-1}\cdot\exp\left(-C\cdot t\right)\! dt\\
 & \leq & \frac{q\cdot\mu}{2\cdot\rho}\cdot\frac{1}{\left(\sqrt{n_{k}\cdot\left(k-n_{k}\right)}\right)^{q}}\int_{\frac{n_{k}^{2}}{\mu/\rho}}^{\infty}t^{q/2}\cdot\exp\left(-C\cdot t\right)\! dt,\end{eqnarray*}
which implies \begin{equation}
\lim_{k\to\infty}\left(\sqrt{k}\right)^{q} \cdot \left(\sum_{i=0}^{n_{k}-1}\sigma_{i}^{2}\right)^{q/2}\cdot I_{2}\left(k\right)=0.\label{eqJ2}\end{equation}
Finally, combining (\ref{eqI})-(\ref{eqJ2}),
we obtain\[
\limsup_{k\to\infty}\sqrt{k}\cdot\left(E\max_{0\leq\ell\leq n_{k}-1}\left(\left|\sigma_{\ell}\right|\cdot \gamma_{m_{\ell,k}}^{\ell}\right)^{q}\right)^{1/q}\leq \frac{\left\Vert \sigma\right\Vert _{2}^{q}}{\left(\sqrt{\mu/\rho}\right)^{q}}.\]
Letting $\rho$ tend to $1$ yields the upper bound in (\ref{eqlemma2}).

For establishing the lower bound in (\ref{eqlemma2}) it suffices to study the case $q=1$. In fact  we have
$$  E\left(\max_{0\leq\ell\leq n_{k}-1}\left|\sigma_{\ell}\right|\cdot \gamma_{m_{\ell,k}}^{\ell}\right)\geq E\left(\left|\sigma_{0}\right|\cdot \gamma_{m_{0,k}}^{0} \right).$$
We use (\ref{eq6N}) and Fatou's Lemma to obtain
$$\liminf_{k\to\infty}\sqrt{k}\cdot E\left(\max_{0\leq\ell\leq n_{k}-1}\left|\sigma_{\ell}\right|\cdot \gamma_{m_{\ell,k}}^{\ell}\right) \geq \frac{\left\Vert \sigma\right\Vert _{2}}{\sqrt{\mu}}, $$
which completes the proof.
\end{proof}
In order to prove the main result given in Theorem \ref{theorem1-1},
we introduce the process $\overline{X}_{n_{k}}$ as follows. For $k\in\mathbb{N}$
let \[
0=t_{0}<t_{1}<\cdots<t_{n_{k}}=1\] 
be the discretization (\ref{discrtization1}) of $\left[0,1\right]$. The
process $\overline{X}_{n_{k}}$ is given by $\overline{X}_{n_{k}}\left(0\right)=X\left(0\right)$
and for $t\in\left[t_{\ell},t_{\ell+1}\right]$\begin{equation}
\overline{X}_{n_{k}}\left(t\right)=\overline{X}_{n_{k}}\left(t_{\ell}\right)+a\left(t_{\ell},\overline{X}_{n_{k}}\left(t_{\ell}\right)\right)\cdot\left(t-t_{\ell}\right)+ \sigma_{\ell}\cdot\left(W\left(t\right)-W\left(t_{\ell}\right)\right).\label{milst}
\end{equation}
Note that $\overline{X}_{n_{k}}$ coincides with the Euler scheme (\ref{euler1})
at the discretization points $t_{\ell}$. Instead of estimating $X-\widehat{X}_{k}^{\dagger}$
directly, we consider $X-\overline{X}_{n_{k}}$, as well as $\overline{X}_{n_{k}}-\widehat{X}_{k}^{\dagger}$
separately. From Proposition 3 in \cite{Slassi:12} we know that 
\begin{equation}
 \left(E\bigl\Vert X-\overline{X}_{n_{k}}\bigr\Vert_{L_{\infty}\left[0,1\right]}^{q}\right)^{1/q} \leq C\cdot \frac{1}{n_{k}}.\label{Eqn7}
\end{equation}
From this and (\ref{asymto1}) it follows that 
\begin{equation}
\lim_{k\to\infty}\sqrt{k}\cdot \left(E\bigl\Vert X-\overline{X}_{n_{k}}\bigr\Vert_{L_{\infty}\left[0,1\right]}^{q}\right)^{1/q} =0,\label{Eqn8}
\end{equation}
and so $\left(E\left\Vert \overline{X}_{n_{k}}-\widehat{X}_{k}^{\dagger}\right\Vert _{L_{\infty}\left[0,1\right]}^{q}\right)^{1/q}$
is the  asymptotically dominating term.

\textbf{Proof of Theorem \ref{theorem1-1}}. In view of the lower bound in Theorem \ref{theorem1-2} it suffices to show
\begin{equation}
\limsup_{k\to\infty}\sqrt{k}\cdot\left(E^{*}\left\Vert X-\widehat{X}_{k}^{\dagger}\right\Vert _{L_{\infty}\left[0,1\right]}^{q}\right)^{1/q}\leq \left(E\left(\tau_{1,1}\right)\right)^{-1/2}\cdot \left\Vert \sigma\right\Vert _{2}.\label{rev4}
\end{equation}
For $t\in\left[t_{\ell},t_{\ell+1}\right]$ we have
$$ \left|\overline{X}_{n_{k}}\left(t\right)-\widehat{X}_{k}^{\dagger}\left(t\right)\right| = \left|\sigma_{\ell}\cdot\left(W^{\ell}\left(t\right)-\widehat{W}_{m_{\ell,k}}^{\ell}\left(t\right)\right)\right|.$$
Thus 
\begin{equation}
\left\Vert \overline{X}_{n_{k}}-\widehat{X}_{k}^{\dagger}\right\Vert _{L_{\infty}\left[0,1\right]} =  \max_{0\leq\ell\leq n_{k}-1}\left(\left|\sigma_{\ell}\right|\cdot\sup_{t_{\ell}\leq t\leq t_{\ell+1}}\left|W^{\ell}\left(t\right)-\widehat{W}_{m_{\ell,k}}^{\ell}\left(t\right)\right|\right).\label{rev5} 
\end{equation}
Then, the estimate (\ref{rev4}) is a direct consequence of (\ref{Eqn8}) together with the equation (\ref{rev5}), (\ref{Wl1}) and Lemma  \ref{lemma2}. 

\textbf{Proof of Theorem \ref{theorem1-2}}
The upper bound in (\ref{appro2}) is a direct
consequence from (\ref{statement2}). For establishing the lower bound it suffices to study the case $q=1$. For $k\in \mathbb{N}$ take $n_{k}\in \mathbb{N}$ such that
\begin{equation}\label{T0} 
\lim_{k\to \infty}\frac{n_{k}}{k}=0\quad\mathrm{and}\quad  \lim_{k\to \infty}\frac{\sqrt{k}}{n_{k}}=0.
\end{equation}
Let $$ \bar{t}_{\ell}=\frac{\ell}{n_{k}}$$ for $\ell=0,\ldots,n_{k},$ and consider the process $\overline{X}_{n_{k}}$ for this discretization; see (\ref{milst}).
At first, by Minkowski's inequality and  (\ref{Eqn7}) we have for every approximation $\widehat{X}_{k} \in \mathfrak{N}_{k,r}$
\begin{equation}\label{T1}
E\left\Vert X-\widehat{X}_{k}\right\Vert _{L_{\infty}\left[0,1\right]}\geq E\left\Vert \overline{X}_{n_{k}}-\widehat{X}_{k}\right\Vert _{L_{\infty}\left[0,1\right]}-C/n_{k}.
\end{equation}
For a fixed  $\omega\in\Omega$ let $\widehat{X}_{k}\left(\omega\right) \in  \Phi_{k,r}$ be given by $$\widehat{X}_{k}\left(\omega\right) =\sum_{j=1}^{k}\mathbf{1}_{\left]t_{j-1},\; t_{j}\right]}\cdot\pi_{j}.$$  Let 
$$ D\left(\widehat{X}_{k}\left(\omega\right)\right) = \{t_{j} : j=0,\cdots,k\}$$ 
be the set of knots used by $\widehat{X}_{k}\left(\omega\right)$, and put 
$$ d_{\ell-1}= \sharp \left(D\left(\widehat{X}_{k}\left(\omega\right)\right)\cap \left]\bar{t}_{\ell-1},\bar{t}_{\ell}\right[\right),\qquad \ell =1,\cdots,n_{k}.$$
We refine the corresponding partition to a partition $$0=\tilde{t}_{0}<\cdots<\tilde{t}_{\tilde{k}}=1,$$that contains all the points $\ell/n_{k}$, and we define the polynomials $\tilde{\pi}_{j} \in \Pi_{r}$ by $$ \widehat{X}_{k}\left(\omega\right) =\sum_{j=1}^{\tilde{k}}\mathbf{1}_{\left]\tilde{t}_{j-1},\;\tilde{t}_{j}\right]}\cdot\tilde{\pi}_{j}.$$ Furthermore, for $t \in \left]\tilde{t}_{j-1}, \tilde{t}_{j}\right]\subseteq \left]\bar{t}_{\ell-1}, \bar{t}_{\ell}\right]$ we define 
$\bar{\pi}_{j} \in \Pi_{r}$ by 
$$  \tilde{\pi}_{j}\left(t\right)= \overline{X}_{n_{k}}\left(\bar{t}_{\ell-1},\omega\right)+ a \left(\bar{t}_{\ell-1},\overline{X}_{n_{k}}\left(\bar{t}_{\ell-1},\omega\right)\right)\cdot\left(t-\bar{t}_{\ell-1}\right)+\sigma_{\ell-1} \cdot\left( \bar{\pi}_{j}\left(t\right)-W\left(\bar{t}_{\ell-1},\omega\right)\right).$$ 
Put $$ \bar{f}=\sum_{j=1}^{\tilde{k}}\mathbf{1}_{\left]\tilde{t}_{j-1},\; \tilde{t}_{j}\right]}\cdot \bar{\pi}_{j}.$$  
Then, we have $$\left\Vert \overline{X}_{n_{k}}\left(\omega\right)-\widehat{X}_{k}\left(\omega\right)\right\Vert _{L_{\infty}\left[0,1\right]}\geq \max_{1\leq \ell \leq n_{k}}\left(\left|\sigma_{\ell-1}\right| \cdot \sup_{\bar{t}_{\ell-1}<t \leq \bar{t}_{\ell}}\left|W\left(t,\omega\right)-\bar{f}\left(t\right)\right|\right).$$  
Note that there exists an $\ell_{0} = \ell_{0}\left(\omega\right)\in \{1,\cdots,n_{k}\}$, so that 
 $$d_{\ell_{0}-1}\leq m_{\ell_{0}-1,k} + 2 .$$  To see this, suppose that $$ d_{\ell-1} > m_{\ell-1,k} + 2 \quad \forall \ell \in \{1,\cdots,n_{k}\}.$$ This implies
$$ k\geq \sum_{\ell= 1}^{n_{k}}d_{\ell-1} > \sum_{\ell=1}^{n_{k}}\left(m_{\ell-1,k} + 2\right) \geq k-n_{k} + 2n_{k} =k+n_{k},$$ which leads to a contradiction.
Hence we a.s. have
\begin{equation}\label{T2}
\sup_{\bar{t}_{\ell_{0}-1}<t \leq \bar{t}_{\ell_{0}}}\left|W\left(t\right)-\bar{f}\left(t\right)\right|\geq \inf_{\varphi\in\Phi_{d_{\ell_{0}-1},r}} \left\Vert W-\varphi\right\Vert_{L_{\infty}\left[\bar{t}_{\ell_{0}-1},\bar{t}_{\ell_{0}}\right]}=\gamma_{d_{\ell_{0}-1}}^{\ell_{0}-1}\geq \gamma_{m_{\ell_{0}-1,k}+2}^{\ell_{0}-1}
\end{equation}  
by (\ref{Wl1}). Hence we use (\ref{rev2}), (\ref{eq6N}), (\ref{T1}) and (\ref{T2}) to obtain  $$\liminf_{k\to\infty}\sqrt{k}\cdot E\left\Vert X-\widehat{X}_{k}\right\Vert _{L_{\infty}\left[0,1\right]}\geq \left(E\tau_{1,1}\right)^{-1/2}\cdot\left\Vert \sigma\right\Vert _{2} $$ by Fatou's Lemma. This completes the proof of Theorem \ref{theorem1-2}.\\\\
\textbf{Acknowledgment}
We would like to thank the referees for numerous useful comments, which led to an improvement of the presentation.
 
\bibliographystyle{plain}
\bibliography{references}

\end{document}